\documentclass[12pt, a4paper]{amsart}
\usepackage{amsmath}
\usepackage{geometry,amsthm,graphics,tabularx,amssymb,shapepar}
\usepackage{amscd}
\usepackage[all,2cell,dvips]{xy}

\newcommand{\End}{{\mathrm{End}}}

\newcommand{\GL}{{\mathrm{GL}}}

\newcommand{\Ind}{{\mathrm{Ind}}}

\newcommand{\SL}{{\mathrm{SL}}}

\newcommand{\SO}{{\mathrm{SO}}}

\newcommand{\vsp}{{\vspace{0.2in}}}

\newcommand{\Tr}{\operatorname{Tr}}

\newcommand{\diag}{\operatorname{diag}}

\newcommand{\od}{\operatorname{d}}

\newcommand{\oO}{\operatorname{O}}

\newcommand{\g}{\mathfrak g}

\renewcommand{\k}{\mathfrak k}

\newcommand{\p}{\mathfrak p}

\renewcommand{\t}{\mathfrak t}

\newcommand{\gl}{\mathfrak g \mathfrak l}

\newcommand{\Z}{\mathbb{Z}}
\newcommand{\C}{\mathbb{C}}
\newcommand{\R}{\mathbb R}

\renewcommand{\S}{\mathbf S}

\newcommand{\F}{\mathbf{F}}

\newcommand{\la}{\langle}
\newcommand{\ra}{\rangle}

\newcommand{\be}{\begin {equation}}
\newcommand{\ee}{\end {equation}}
\newcommand{\bee}{\begin {equation*}}
\newcommand{\eee}{\end {equation*}}

\theoremstyle{Theorem}

\newtheorem{thm}{Theorem}[section]

\theoremstyle{Theorem}
\newtheorem{lem}{Lemma}[section]

\theoremstyle{Theorem}
\newtheorem{prp}{Proposition}[section]

\theoremstyle{Plain}

\theoremstyle{Definition}

\begin{document}

\title[positivity of principal matrix coefficients]{Positivity of principal matrix coefficients of principal series representations of $\GL_n(\R)$}

\author{Yangyang Chen}
\address{Academy of Mathematics and Systems Science\\
Chinese Academy of Sciences\\
Beijing, 100190, P.R. China}
\email{chenyy@amss.ac.cn}



\begin{abstract}
Let $G=\GL_n(\R)$, with the usual Cartan decomposition $G=K\exp(\p_0)$ and the usual Iwasawa decomposition $G=NAK$. We determine the image of $\exp(\p_0)$ under the projection of $G$ to $K$ through the Iwasawa decomposition. As an application, we prove a positivity result
about the matrix coefficients of principal series representations of $G$.
\end{abstract}

\maketitle

\section{Introduction}
\label{sec:intro}

 Let $G=\GL_n(\R)$ ($n\geq 1$). It has an Iwasawa decomposition
 \[
   G=NAK,
 \]
 where $N$ is the subgroup of $G$ of upper triangular unipotent matrices, $A$ is the subgroup of diagonal matrices with positive entries and $K$ is the group of orthogonal matrices in $G$.

In this paper, we will prove some positivity results about the matrix coefficients of principal series representations of $G$. To be precise, let $M$ be the subgroup of $G$ consists of all diagonal matrices in $G$ with diagonal entries $\pm 1$.  Let $\delta: M\rightarrow \C^\times$ and $\nu: A\rightarrow \C^\times $ be two characters. Define
\[
  I_{\delta\otimes \nu}:=\Ind_{MAN}^G (\delta\otimes \nu) \quad(\textrm{normalized smooth induction}),
\]
which is called a principal series representation of $G$. It is well known hat $I_{\delta\otimes \nu}$ has a unique lowest $K$-type $\alpha$, and $\alpha$ occurs with multiplicity one in $I_{\delta\otimes \nu}$ (see \cite[Theorem 4.9]{Vo1}). By duality, $\alpha^\vee$ is the unique lowest $K$-type of $I_{\delta\otimes \nu}^\vee$, which occurs with multiplicity one. Here and henceforth, ``$\,^\vee$'' indicates the contragradient representation. As is quite often, we do not distinguish a representation and its underlying space. View $\alpha$ as a subspace of $I_{\delta\otimes \nu}$, and $\alpha^\vee$ as a subspace of $I_{\delta\otimes \nu}^\vee$ (we assume the injection is compatible with the natural pairings, see \cite[Section 3]{TH}). We define the principal matrix coefficient of  $I_{\delta\otimes \nu}$ to be the function
\[
  g\mapsto \psi_{\delta\otimes \nu}(g):=\sum_{i=1}^r \la g. u_i, v_i\ra
\]
on $G$, where $\{u_i\}_{1\leq i\leq r}$ is a basis of $\alpha$, and $\{v_i\}_{1\leq i\leq r}$ is the dual basis of $\alpha^\vee$, and $\la\,,\,\ra$ stands for the natural paring between $I_{\delta\otimes \nu}$ and $I_{\delta\otimes \nu}^\vee$. This principal matrix coefficient does not depend on the choice of the basis $\{u_i\}_{1\leq i\leq r}$. Actually, it follows easily that
\[
\psi_{\delta\otimes \nu}=\Tr\circ\phi_{I_{\delta\otimes \nu},\alpha},
\]
where $\phi_{I_{\delta\otimes\nu},\alpha}: G\rightarrow \End_{\C}(\alpha)$ is the matrix coefficient of the representation $I_{\delta\otimes\nu}$ with respect to the $K$-type $\alpha$, which will be defined in the next section.

Now we can state the main result of this paper.

\begin{thm}
\label{thm2}
With the notations as above, assume that $I_{\delta\otimes \nu}$ has real infinitesimal character, or equivalently, the image of $\nu$ is contained in the set of positive real numbers. Then
\[
  \psi_{\delta\otimes \nu}(g)>0
\]
for all $g\in \exp(\p_0)$.
\end{thm}

Recently, Sun \cite{S3} proves the non-vanishing hypothesis at infinity for Rankin-Selberg convolutions, a long awaited problem which appears in the arithmetic study of special values of L-functions. A key ingredient in his proof is \cite[Theorem 1.5]{S2}, which asserts the positivity of some matrix coefficients associated to the lowest $K$-type of an irreducible unitary representation with nonzero cohomology. Our results in this paper can serve as a weak analogue for principal series representations in the $\GL_n(\R)$ case. Nevertheless, this weak version is often enough for applications to branching problems, see \cite{Li,HL} for examples.

This paper is arranged as follows: Firstly, we recall briefly the definition of matrix coefficients and some basic facts of it in the next section. Then we will prove an elementary result in Section 3, which is vital to the proof of Theorem \ref{thm2}. In Section 4, we will analyze the lowest $K$-types of the principal series representations. Particularly, we find that the lowest $K$-types decomposes nicely under the action of $M$. Finally, with these preparations in hand, we will complete the proof of Theorem \ref{thm2} in the last section.\\

 \noindent{\bf Acknowledgements}: The author would like to thank Binyong Sun for initiating this work and for insightful comments on the proof of the main result.

\vsp
\section{Notion of matrix coefficients}
\label{sec:matcoe}

In this section, we will recall briefly the notion of matrix coefficients. We will follow Sun \cite{S2} with minor modifications. The principal series representations of $\GL_{n}(\R)$ have nice properties, namely, they belong to the category of the so called ``Casselman-Wallach representations". We will define the matrix coefficients for these representations.

In this section only, we let $G$ be a general real reductive Lie group. Fix a maximal compact subgroup $K$, and denote by $\g$ the complexified Lie algebra of $G$. By a representation of $G$, we mean a continuous linear action of $G$ on a complete, locally convex, Hausdorff, complex topological vector space. A representation $V$ is called a Casselman-Wallach representation if the space $V$ is Fr\'{e}chet, and the representation is smooth and of moderate growth, and Harish-Chandra. For more details about these representations, we refer the reader to \cite{Ca} and \cite{Wa}. Let $\pi$ be a Casselman-Wallach representation, $\alpha$ a $K$-type in $\pi$ with multiplicity one, we define the matrix coefficient of type $\alpha$ by
\[
\phi_{\pi,\alpha}(g):= p_{\alpha}\circ\pi(g)\circ j_{\alpha}, \quad \text{$g\in G$},
\]
where $j_{\alpha}$ is the $K$-equivalent embedding of $\alpha$ into $\pi$, and $p_{\alpha}$ is the continuous $K$-equivalent linear map from $\pi$ onto $\alpha$ such that
\[
p_{\alpha}\circ j_{\alpha}=\text{identity map on $\alpha$}.
\]
The matrix coefficient $\phi_{\pi,\alpha}$ is a real analytic function on $G$ with values in $\End_{\C}(\alpha)$. It is independent of the choice of $j_{\alpha}$ and is determined by the isomorphic class of $\pi$.

A $(\g, K)$-module is called a Harish-Chandra module if it is of finite length. By definition, the underlying $(\g, K)$-module of a Casselman-Wallach representation is necessarily a Harish-Chandra module. The Casselman-Wallach globalization theorem(\cite{Ca}, \cite{Wa} and \cite{BK}) essentially asserts that the category of Casselman-Wallach representations is equivalent to the category of Harish-Chandra modules. Thus the matrix coefficient $\phi_{\pi,\alpha}$ defined above is actually determined by the underlying $(\g, K)$-module of $\pi$. One can define the matrix coefficient $\phi_{V,\alpha}$ for a Harish-Chandra module $V$ by taking globalization as in \cite{S2}. For more details about globalization, we refer the reader to \cite{BK} and \cite[Chapter 11]{Wa}.

A general problem about matrix coefficient is to find a formula similar to Harish-Chandra's integral formula, for any irreducible admissible representation with respect to any lowest $K$-type of it.
Sun in \cite{S1} gives an integral formula for $\pi$ a cohomologically induced representation and $\alpha$ in the bottom layer of $\pi$, and in that same paper he explained why his formula is enough for the general problem. However, we will only consider principal series representations in this paper, and Harish-Chandra's Eisenstein integral formula will be enough. We will recall it later in this paper and use it to derive our Theorem \ref{thm2}.

Finally, we recall a general fact \cite[Facts 2.1]{S2} about matrix coefficients which will be used in our proof of Theorem \ref{thm2}.
\begin{prp}
\label{reduc}
Let $V$ be a Harish-Chandra module, $\alpha$ a $K$-type occurring in $V$ with multiplicity one, then
\[
\phi_{V,\alpha}=\phi_{V_{1},\alpha}
\]
where $V_{1}$ is the irreducible sub-quotient of $V$ containing the $K$-type $\alpha$.
\end{prp}

\vsp
\section{Image of the split part in $K$}
\label{sec:image}

With the notations as in the introduction, the group $G$ also has a Cartan decomposition
 \[
   G=K\exp(\p_0),
 \]
 where $\p_0$ is the space of symmetric matrices in $\gl_n(\R)$. Note that $\exp(\p_0)$ equals to the set of positive definite symmetric matrices in $G$. We call it the split part of $G$. Let $\kappa:G\rightarrow K$ be the projection of $G$ to $K$ through the Iwasawa decomposition $G=NAK$. In this section, we will determine the image of the split part in $K$ under the projection $\kappa$.

 Let $M(n,\R)$ denote the set of square matrices of order $n$. For $B=(B_{i,j})\in M(n,\R)$, we denote by $B_{r}=det((B_{i,j})_{1\leq i,j\leq r})$ and $B_{r}^{'}=det((B_{i,j})_{n-r+1\leq i,j\leq n})$ for $1\leq r\leq n$.
Firstly, we need an simple fact from linear algebra.

\begin{lem}
\label{sym}
For each $B=(B_{i,j})\in M(n,\R)$, satisfying $B_{r}>0$ for $1\leq r\leq n$, there exists an upper-triangular matrix $C$ with each diagonal entry equals one, such that $BC$ is a symmetric matrix.
\end{lem}

\begin{proof}
This lemma can be proved by induction on the order $n$ and the technique of matrix blocking. We omit the details.
\end{proof}

\begin{prp}
\label{mina}
With notations as above, the following three sets are equal:\\
(a) $\kappa(\exp(\p_{0}))$, the image of $\exp(\p_{0})$ under the projection $\kappa$,\\
(b) $\{B\in K: B_{r}>0$, \textrm{for} $ 1\leq r\leq n\}$,\\
(c) $\{B\in K: B_{r}^{'}>0$, \textrm{for} $1\leq r\leq n\}$.
\end{prp}

\begin{proof}
We will firstly show that (a) equals (b). For $g\in \exp(\p_{0})$, write $g=nak=bk$ in the Iwasawa decomposition, where $b=na$ is an upper-triangular matrix with each diagonal entry positive.
For each $1\leq r\leq n$, denote
\[
  g^{-1}=(p_{i,j})_{1\leq i,j\leq n}=\left[\begin{array}{cc}
                   g^{-1}_{11}&g^{-1}_{12}\\
                   g^{-1}_{21}&g^{-1}_{22}\\
              \end{array}\right],
\]
where $g^{-1}_{11}=(p_{i,j})_{1\leq i,j\leq r}$, and $g^{-1}_{12}, g^{-1}_{21}, g^{-1}_{22}$ have the obvious meaning. Similar notations apply to the matrices $b^{-1}$ and $k^{t}$ (the transpose of $k$). Write the matrix multiplication $g^{-1}=k^{t}b^{-1}$ in blocks, we have
\[
g^{-1}_{11}=k^{t}_{11}b^{-1}_{11}+k^{t}_{12}b^{-1}_{21}=k^{t}_{11}b^{-1}_{11}.
\]
Taking determinants on both sides and noting that each leading principal minor of a positive definite matrix is positive, we conclude that $k^{t}$ lies in (b). Clearly, then $k$ also lies in (b).

For the converse, let $k\in K$ satisfying the conditions in (b). By Lemma \ref{sym}, there exists an upper-triangular matrix $b$ with each diagonal entry equals one, such that $k^{t}b$ is symmetric. Since $k^{t}b$ also satisfying the conditions in (b), it must be positive definite. Now $b^{-1}k$ lies in $\exp(\p_{0})$ and $\kappa(b^{-1}k)=k$, this finishes the proof that (a) equals (b).

For the proof of (a) equals (c), we use the fact that a real symmetric matrix is positive definite if and only if it satisfies the conditions in (c). The other steps are similar as in the proof of (a) equals (b).
\end{proof}

We end this section with a remark. The above proposition was already observed by Sun \cite[Lemma 3.4]{S2}  in the case of $\SL(2,\R)$.

\vsp
\section{Lowest $K$-types of principal series representations}
\label{sec:lowest}

In this section, we will determine the lowest $K$-types of the principal series representations of $G=\GL_{n}(\R)$.

Firstly, we recall the definition of lowest $K$-types by Vogan. For a general real reductive group $G$, fix a maximal compact subgroup $K$. Let $T_{0}$ be a maximal torus with Lie algebra $\t_{0}$. Denote $\triangle(\k,\t)$ for the corresponding roots, and let $\triangle^{+}$ be a positive system. For a finite dimensional irreducible representation $\mu$ of $K$, define the norm $|\mu|=\la\gamma+2\rho_{c},\gamma+2\rho_{c}\ra^{1/2}$, where $\gamma$ is a highest weight of $\mu$, $\rho_{c}$ is the half sum of the positive $\t$ roots and for the inner product $\la\,,\,\ra$ we choose the one that was incorporated in the definition of a reductive group (\cite[Page 244]{KV}). Let $\pi$ be a continuous representation of $G$, the lowest $K$-types of $\pi$ are defined to be the $K$-types that occur in $\pi$ with minimal norms. For more details, see \cite[Chapter X]{KV}. It should be pointed out that the lowest $K$-types are independent of the various choices of data.

 Now we turn to our $\GL_{n}(\R)$ case, with notations as in the introduction. The principal series representations are the ones induced from minimal parabolic subgroups, explicitly,
\begin{eqnarray*}
  I_{\delta\otimes\nu}&=&\Ind_{MAN}^{G}(\delta\otimes\nu)\\
  &=&\{f\in C^{\infty}(G;\C): f(manx)= \delta(m)\nu(a)\prod_{i=1}^{n}a_{i}^{\frac{n-2i+1}{2}}f(x),\\
 && m\in M, a=\diag(a_{1},...,a_{n})\in A, n\in N, x\in G\},
\end{eqnarray*}
with $G$ acts via right translation, where $\delta$ is a character of $M$, $\nu$ is a character of $A$. It is well known that $I_{\delta\otimes\nu}$ has a real infinitesimal character (in the sense of Vogan) if and only if the image of $\nu$ is contained in the set of positive real numbers, we define such characters to be the real characters of $A$.

 The Weyl group of $A$ is defined to be
 \[
 W:=N_{K}(A)/Z_{K}(A).
 \]
 It acts on $M, \hat{M}, A, \hat{A}$ in an obvious way, where $\hat{M}$ and $\hat{A}$ denote the sets of characters of $M$ and $A$ respectively. Let $S_{n}$ denote the $n$-th symmetric group, we have a natural embedding
\[
\begin{array}{rcl}
 \iota:S_{n} &\hookrightarrow & K,\\
                 \sigma &\mapsto& \Sigma_{i}E_{\sigma(i),i},
\end{array}
\]
where $E_{i,j}$ denotes the usual matrix unit. By a direct computation,
\[
N_{K}(A)=\iota(S_{n})\ltimes Z_{K}(A),
\]
 thus $W\cong S_{n}$.

For $\delta\in\hat{M}$, let $A(\delta)$ denote the lowest $K$-type of $I_{\delta\otimes\nu}$. It does not depend on $\nu$, since
\[
I_{\delta\otimes\nu}|_{K}\cong \Ind_{M}^{K}\delta,
 \]
following from the Iwasawa decomposition of $G$. Our main goal in this section is to determine $A(\delta)$ for each $\delta\in\hat{M}$. Recall that $M=\{\pm 1\}^{n}$, embedded in $K$ in the usual way.
Denote by $\delta_{r}=(\underbrace{1\cdots1}_{r}~0\cdots0)$ the character of $M$, which sends the first $r$'s $-1$ of $M$ to $-1$ and others to $1$. It is enough to consider $A(\delta_{r})$ for
$0\leq r\leq n$, since each $\delta\in\hat{M}$ lies in the Weyl group orbit of some $\delta_{r}$ and $\Ind_{M}^{K}\delta\cong \Ind_{M}^{K}(\sigma.\delta)$  for each $\sigma\in W$ by \cite[Lemma 2.27]{BZ}.

\begin{prp}
\label{lo}
With notations as above, \\
\begin{equation*}
  A(\delta_{r})=\left\{
                     \begin{array}{ll}
                        \bigwedge^{r}\C^{n},\quad&\textrm{if ~ $0\leq r\leq [n/2]$,}\\
                         \det \otimes\bigwedge^{n-r}\C^{n}, \quad&\textrm{if ~ $[n/2]+1\leq r\leq n$.}\smallskip\\

                     \end{array}
              \right.
\end{equation*}
where the $\oO(n)$-action on $\bigwedge^{r}\C^{n}$ is standard. Moreover, $A(\delta_{r})\cong\bigoplus_{\delta\in W.\delta_{r}}\delta$ as representations of $M$.
\end{prp}

\begin{proof}
This result may be known to experts. We give a proof here for the lack of reference.
By Frobenius reciprocity law, the $K$-types occurring in $\Ind_{M}^{K}\delta_{r}$ are just the ones whose restriction to $M$ contains the $M$-type $\delta_{r}$. We will start from the $K$-types with small norms and analyze their restriction to $M$. For clarity, we treat the case separately for $n$ even and odd.

\noindent Case a : $n=2m+1, m\geq 0$

In this case, $\oO(n)=\SO(2m+1)\times\{\pm I_{2m+1}\}$ as a direct product. Each irreducible representation of $\SO(2m+1)$ corresponds to two irreducible ones of $\oO(2m+1)$. Choose a maximal torus $T_{0}=\SO(2)^{m}$, embedded in $\SO(2m+1)$ in the usual way. Since $\SO(2)\cong \S^{1}$, the complex number with absolute value 1, one has $\hat{T_{0}}\cong\Z^{m}$. Form the root system $\triangle(\k,\t)$ and choose as in \cite[Chapter IV]{KN}
\[
\triangle^{+}=\{e_{i}\pm e_{j}, i<j\}\cup\{e_{k}\}.
\]
Then the highest weights of irreducible representations of $\SO(2m+1)$ are characterized by a decreasing sequence of $m$ nonnegative integers. For the highest weight $\lambda=(b_{1}\cdot\cdot\cdot b_{m})$, the corresponding representation $\phi_{\lambda}$ has norm
\[
|\phi_{\lambda}|=\sqrt{\Sigma_{i=1}^{m}(2m-2i+1+b_{i})^{2}}.
\]
 Obviously, the norm is small provided that each component $b_{i}$ is small. By \cite[Page 90]{KN}, the representation $\phi_{\lambda_{r}}$ corresponding to
$\lambda_{r}=(1\cdot\cdot\cdot1~0\cdot\cdot\cdot0)$ (with $r$ ones) is $\bigwedge^{r}\C^{2m+1}$, for $0\leq r\leq m$. Let$f_{k}$ denote the basis vector of $\C^{2m+1}$ with $1$ at the $k$-th component and $0$ at other components, then clearly $M$ acts on $\bigwedge^{r}\C^{2m+1}$ with the $1$-dim subspace spanned by $f_{1}\wedge\cdot\cdot\cdot\wedge f_{r}$ invariant, the corresponding character is just $\delta_{r}$ described above. Actually, it follows obviously that $\phi_{\lambda_{r}}\cong\bigoplus_{\delta\in W.\delta_{r}}\delta$ as representations of $M$. Now we claim that $A(\delta_{r})=\phi_{\lambda_{r}}$, for
$0\leq r\leq m$. To see this, it is enough to show that $\delta_{r}$ does not occur in $\phi_{(b_{1}\cdot\cdot\cdot b_{m})}$ or $\det\otimes\phi_{(b_{1}\cdot\cdot\cdot b_{m})}$ for $|\phi_{(b_{1}\cdot\cdot\cdot b_{m})}|\leq|\phi_{\lambda_{r}}|$, unless $(b_{1}\cdot\cdot\cdot b_{m})=\lambda_{r}$. From the formula of norms, it follows that $\Sigma_{i=1}^{m}b_{i}\leq r$. If the sum equals $r$, again by the norm formula, it forces
$(b_{1}\cdot\cdot\cdot b_{m})=\lambda_{r}$. If the sum is strictly less than $r$, by embedding $\phi_{(b_{1}\cdot\cdot\cdot b_{m})}$ into
\[
(\C^{n})^{b_{1}-b_{2}}\otimes(\wedge^{2}\C^{n})^{b_{2}-b_{3}}\otimes\cdot\cdot\cdot\otimes(\wedge^{m}\C^{n})^{b_{m}},
\]
 we see that $\delta_{j}$ occurs in
$\phi_{(b_{1}\cdot\cdot\cdot b_{m})}$ with $j$ at most
\[\Sigma_{i=1}^{m-1}i(b_{i}-b_{i+1})+mb_{m}=\Sigma_{i=1}^{m}b_{i},
\]
 which is strictly less than $r$, finishing the proof of the claim.

For $m+1\leq r\leq n=2m+1$, it is obvious that $\delta_{r}$ occurs in $\det\otimes\phi_{\lambda_{n-r}}$ and $\det\otimes\phi_{\lambda_{n-r}}\cong\bigoplus_{\delta\in W.\delta_{r}}\delta.$
The same argument shows that $A(\delta_{r})=\det\otimes\phi_{\lambda_{n-r}}$, finishing the proof of odd case.\\
Case b : $n=2m, m\geq 1$

In this case, $\oO(n)=\oO(2m)=\SO(2m)\rtimes\{1,r_{n}\}$ as a semidirect product, where $r_{n}=\diag(1,...,1,-1)$. Again we choose the maximal torus $T_{0}=\SO(2)^{m}$. Now the positive roots are
\[
\triangle^{+}=\{e_{i}\pm e_{j}, i<j\}.
\]
The highest weights of irreducible representations of $\SO(2m)$ are characterized by
decreasing sequences of $m$ integers with the condition that the last two has a nonnegative sum. Any irreducible representation of $\oO(n)$ restricted to $\SO(n)$ is either irreducible or is the sum of two irreducible $\SO(n)$-representations. In the first case, the corresponding highest weights ends with 0 and in the second case, the two highest weights differ only in the last component, actually, if one is $\gamma_{m}$ then the other is $-\gamma_{m}$, of course $\gamma_{m}\neq 0$ \cite{Vo1}.

For $\lambda=(b_{1}\cdot\cdot\cdot b_{m})$, the corresponding representation $\phi_{\lambda}$ of $\SO(n)$ has norm
\[
 |\phi_{\lambda}|=\sqrt{\Sigma_{i=1}^{m}(2m-2i+b_{i})^{2}}.
\]
By \cite[Page 90]{KN}, $\SO(n)$ acts on $\bigwedge^{r}\C^{n}$ irreducibly for $0\leq r\leq m-1$ and the corresponding highest weights are $(1\cdot\cdot\cdot1~0\cdot\cdot\cdot0)$ (with $r$ ones). Then $\oO(n)$ also acts irreducibly on $\bigwedge^{r}\C^{n}$. What is different from the odd case is that the action of $\SO(n)$ on $\bigwedge^{m}\C^{n}$ is reducible. Nevertheless, the $\oO(n)$-action is still irreducible. Actually, $\bigwedge^{m}\C^{n} = \phi_{(1\cdot\cdot\cdot1~1)}\oplus\phi_{(1\cdot\cdot\cdot1~-1)}$ as representations of $\SO(n)$. If the $\oO(n)$-action is reducible, say $\bigwedge^{m}\C^{n}=V_{1}\oplus V_{2}$, where $V_{1}, V_{2}\in\widehat{\oO(n)}$. Then $V_{i}$ must be irreducible under the action of $\SO(n)$, so the two highest weights both ends with $0$, a contraction. Now we can proceed as in the odd case to complete the proof.
\end{proof}

For more general results about the lowest $K$-types in the $\GL_{n}$ case, see \cite[Theorem 4.9]{Vo1}. Since $\GL_{n}(\R)$ is quasi-split, the $K$-types $A(\delta_{r})$ is small in the sense of Vogan, and the decomposition in Proposition \ref{lo} is a general property of small $K$-types. See \cite[Theorem 6.4]{Vo} for more details about small $K$-types.

\vsp
\section{Proof of the main theorem}
\label{sec:proof}

Now we can prove the main result of this paper. We restate it here for the convenience of the reader.

\begin{thm}
\label{main}
With the notations as in the introduction, for each real $\nu\in\hat{A}$ and $\delta\in\hat{M}$, the matrix coefficient $\phi_{\delta\otimes\nu,\alpha}(x)$ of the principal series representations $I_{\delta\otimes\nu}$ with respect to the lowest $K$-type $\alpha$ has positive trace as linear operators on $\alpha$, for each $x\in \exp(\p_{0})$.
\end{thm}

We have written $\phi_{\delta\otimes\nu,\alpha}$ instead of $\phi_{I_{\delta\otimes\nu},\alpha}$ in the above theorem for simplicity. In \cite[Lemma 3.1]{S2}, it was shown that $\phi_{\delta\otimes\nu,\alpha}$ has Hermitian operator values on $\exp(\p_{0})$, and the positive-definite property was proved when $G$ is complex or has real rank one. Our result points out that it must have positive trace in the $\GL_{n}(\R)$ case. To prove Theorem \ref{main}, firstly, we use a general result about principal series representations to reduce the main theorem to the normal case, namely $\delta$ equals $\delta_{r}$ for some $0\leq r\leq n$. Then with the preparations in Section \ref{sec:image} and Section \ref{sec:lowest}, we prove the theorem in the normal case.

\subsection{Reduction to normal case}

We shall use a general result of Harish-Chandra, \cite[Theorem 6.1]{Vo}.

\begin{thm}[Harish-Chandra, Bruhat]
\label{quote}
For $\delta\in\hat{M}$ and $\nu\in \hat{A}$, the representation $I_{\delta\otimes\nu}$ is admissible. The representations $I_{\delta\otimes\nu}$ and $I_{\delta^{'}\otimes\nu^{'}}$ have equivalent composition series if and only if $(\delta^{'}, \nu^{'})=(\sigma.\delta, \sigma.\nu)$ for some $\sigma\in W$.
\end{thm}

Theorem \ref{quote} holds in a more general setting, but we need it only in the $\GL_{n}(\R)$ case. Suppose $\alpha$ is the lowest $K$-type of $I_{\delta\otimes\nu}$, since
\[
I_{\delta\otimes\nu}|_{K}\cong \Ind_{M}^{K}\delta\cong \Ind_{M}^{K}(\sigma.\delta)\cong I_{\sigma.\delta\otimes\sigma.\nu}|_{K}
\]
as representations of $K$, $\alpha$ is also the lowest $K$-type of $I_{\sigma.\delta\otimes\sigma.\nu}$, for each $\sigma\in W$. By Theorem \ref{quote} and Proposition \ref{reduc}, we conclude that
\[
\phi_{\delta\otimes\nu,\alpha}=\phi_{\sigma.\delta\otimes\sigma.\nu,\alpha}.
\]
Since $W\backslash\hat{M}=\{\delta_{r}:0\leq r\leq n\}$ and $W$ leaves the real characters of $A$ invariant, it is enough to prove Theorem \ref{main} for $\delta$ equals some $\delta_{r}$.\\

\subsection{Proof of the normal case}

Since $\alpha$ occurs in $I_{\delta_{r}\otimes\nu}$ with multiplicity one, the matrix coefficient $\phi_{\delta_{r}\otimes\nu,\alpha}$ can be expressed as Harish-Chandra's Eisenstein integral:
\begin{displaymath}
\label{integral}
\phi_{\delta_{r}\otimes\nu,\alpha}(x)=\frac{\deg(\alpha)}{\deg(\delta_{r})}\int_{K}\nu(H(k^{-1}x))H(k^{-1}x)^{\rho}\pi_{\alpha}(k)\circ E_{\delta_{r}}\circ\pi_{\alpha}(\kappa(k^{-1}x))\od\!k,
\end{displaymath}
where $H: G\rightarrow A$ is the projection through the Iwasawa decomposition $G=NAK$,
\[
a^{\rho}:=\prod_{i=1}^{n}a_{i}^{\frac{n-2i+1}{2}}
\]
provided $a=\diag(a_{1},...,a_{n})\in A$, $\pi_{\alpha}$ denotes the $K$-action on $\alpha$, $E_{\delta_{r}}$ denotes the orthogonal projection to the $\delta_{r}$ component of $\alpha$ (we know that $\delta_{r}$ occurs in $\alpha$ with multiplicity one) and $\od\!k$ is the normalized Haar measure on the compact group $K$. Consider the trace, we have
\begin{eqnarray*}
  &&\Tr(\phi_{\delta_{r}\otimes\nu,\alpha}(x))\\
  &=&\frac{\deg(\alpha)}{\deg(\delta_{r})}\int_{K}\nu(H(k^{-1}x))H(k^{-1}x)^{\rho}\Tr(\pi_{\alpha}(k)\circ E_{\delta_{r}}\circ\pi_{\alpha}(\kappa(k^{-1}x)))\od\!k\\
  &=&\frac{\deg(\alpha)}{\deg(\delta_{r})}\int_{K}\nu(H(k^{-1}x))H(k^{-1}x)^{\rho}\Tr(E_{\delta_{r}}\circ\pi_{\alpha}(\kappa(k^{-1}x))\circ\pi_{\alpha}(k))\od\!k\\
  &=&\frac{\deg(\alpha)}{\deg(\delta_{r})}\int_{K}\nu(H(k^{-1}x))H(k^{-1}x)^{\rho}\Tr(E_{\delta_{r}}\circ\pi_{\alpha}(\kappa(k^{-1}x k)))\od\!k
\end{eqnarray*}
Since $\nu$ is real, $\nu(H(k^{-1}x))H(k^{-1}x)^{\rho}>0$. Theorem \ref{main} follows if we show that $\Tr(E_{\delta_{r}}\circ\pi_{\alpha}(\kappa(k^{-1}xk)))>0$ for each $k\in K, x\in \exp(\p_{0})$. Using Proposition \ref{lo}, we can calculate the trace explicitly.

For $0\leq r\leq [n/2]$, we take $\alpha=\bigwedge^{r}\C^{n}$. If $r=0$, the trace is trivially positive. For $1\leq r\leq [n/2]$, using the basis $\{e_{i_{1}}\wedge\cdots\wedge e_{i_{r}}|\{i_{1}\cdots i_{r}\}\in\{1~2\cdots n\}\}$ of $\bigwedge^{r}\C^{n}$,
\[
\Tr(E_{\delta_{r}}\circ\pi_{\alpha}(\kappa(k^{-1}xk)))=(\kappa(k^{-1}xk))_{r},
\]
where the notation is as in Proposition \ref{mina}. Note that $x\in \exp(\p_{0})$ implies that $k^{-1}xk\in \exp(\p_{0})$ for each $k\in K$, thus by Proposition \ref{mina}(b),
\[
\Tr(E_{\delta_{r}}\circ\pi_{\alpha}(\kappa(k^{-1}xk)))>0.
\]

For $[n/2]+1\leq r\leq n$, we take $\alpha=\det\otimes\bigwedge^{n-r}\C^{n}$. If $r=n$, the positivity of the trace follows from the fact that $\det(\kappa(k^{-1}xk))>0$. For $[n/2]+1\leq r< n$, using the basis $\{e_{i_{1}}\wedge\cdots\wedge e_{i_{n-r}}|\{i_{1}\cdots i_{n-r}\}\in\{1~2\cdots n\}\}$ of $\det\otimes\bigwedge^{n-r}\C^{n}$,
\[
\Tr(E_{\delta_{r}}\circ\pi_{\alpha}(\kappa(k^{-1}xk)))=(\kappa(k^{-1}xk))^{'}_{n-r}>0,
\]
by Proposition \ref{mina}(c). Now the proof of Theorem \ref{main} is completed.


\begin{thebibliography}{10}



\bibitem{BK}
J. Bernstein and B. Krotz, \emph{Smooth Frechet globalizations of Harish-Chandra modules.} Israel J. Math. 199(2014), 45-111.

\bibitem{BZ}
I.N. Bernstein and A.V. Zelevinskii, \emph{Representations of the group $GL(n,\F)$, where $\F$ is a non-archimedean local field.} Russian Math. Surveys 31(1976), 1-68.

\bibitem{Ca}
W. Casselman, \emph{Canoncial extensions of Harish-Chandra modules to representations of $G$.} Canad. J. Math. 41(1989), 385-438.

\bibitem{HL}
M. Harris and J.S. Li, \emph{A Lefschetz property for subvarieties of Shimura varieties.} J. Algebraic Geom. 7(1998), 77-122.

\bibitem{KN}
A. Knapp, \emph{Representations of semisimple groups, an overview based on examples.} Princeton University Press, Princeton, 1986.

\bibitem{KV}
A. Knapp and D. Vogan, \emph{Cohomological induction and unitary representations.} Princeton University Press, Princeton, 1995.

\bibitem{Li}
J.S. Li, \emph{Theta liftings for unitary represnetations with non-zero cohomology.} Duke Math. J. 61(1990), 913-937.

\bibitem{S1}
B. Sun, \emph{Matrix coefficients of cohomologically induced representations.} Compos. Math. 143(2007), 201-221.

\bibitem{S2}
B. Sun, \emph{Positivity of matrix coefficients of represnetations with real infinitesimal characters.} Israel J. Math. 170(2009), 395-410.

\bibitem{S3}
B. Sun, \emph{The non-vanishing hypothesis at infinity for Rankin-Selberg convolutions.} J. Amer. Math. Soc. 30(2017), 1-25.

\bibitem{TH}
H. Takahiro, K. Harutaka, M. Tadashi and O. Takayuki, \emph{Matrix coefficients of discrete series representations of SU(3,1).} J. Lie Theory 25(2015), 271-306.

\bibitem{Vo}
D. Vogan, \emph{The algebraic structure of the representations of semisimple Lie groups $I$.} Ann. of Math.(2) 109(1979), 1-60.

\bibitem{Vo1}
D. Vogan, \emph{The unitary dual of GL(n) over an archimedean field.} Invent. Math. 83(1986), 449-505.

\bibitem{Wa}
Nolan R. Wallach, \emph{Real reductive groups \uppercase\expandafter{\romannumeral2}.} Academic Press, Pure and Applied Mathmatics, 1992.

\end{thebibliography}
\end{document}